\documentclass[12pt,leqno]{article}
\usepackage{amsfonts,amsthm,amsmath}

\theoremstyle{plain}
\newtheorem{thm}{Theorem}[section]

\newtheorem{conj}{Conjecture}[section]

\theoremstyle{definition}

\newcommand{\ZZ}{\mathbb{Z}}

\DeclareMathOperator{\tr}{Tr}

\newcommand{\slfrac}[2]{\left.\left.#1\right/#2\right.}

\begin{document}

\title{The McKay-Thompson series of \\ Mathieu Moonshine modulo two}

\author{Thomas Creutzig\thanks{
\noindent Department of Mathematical and Statistical Sciences, University of Alberta,
Edmonton, Alberta  T6G 2G1, Canada. email: creutzig@ualberta.ca},
\ Gerald H\"ohn\thanks{Kansas State University, 138 Cardwell Hall, Manhattan, KS 66506-2602, USA. email: gerald@math.ksu.edu}
\ and 
Tsuyoshi Miezaki\thanks{
Yamagata University, 
1-4-12 Kojirakawa, Yamagata 990-8560, Japan. email: miezaki@e.yamagata-u.ac.jp
}
\footnote{
This work was supported by JSPS KAKENHI Grant Number 22840003, 24740031.
}
}

\date{}


\maketitle

\begin{abstract}
In this note, we describe the parity of the
coefficients of the McKay-Thompson series of Mathieu moonshine.
As an application, we prove a conjecture of 
Cheng, Duncan and Harvey stated in connection with umbral moonshine
for the case of Mathieu moonshine. 
\end{abstract}






\section{Introduction}\label{intro}

In 2010, Eguchi, Ooguri, and Tachikawa~\cite{EOT} discovered a 
phenomenon connecting the Mathieu group $M_{24}$ and the elliptic genus of a K3 surface.
To describe their observation, we let $q=e^{2\pi i \tau}$ and consider the function
\begin{align*}
\Sigma (\tau)&=-8(\,\mu(1/2;\tau)+\mu(\tau/2;\tau)+\mu((\tau+1)/2;\tau))\\
&=q^{-\frac{1}{8}}(-2 + 90 q + 462 q^2 + 1540 q^3 + 4554 q^4 + 11592 q^5 + 27830 q^6 +\cdots)
\end{align*}
where
\begin{align*}
\mu(z;\tau)&\ =\ \frac{ie^{\pi i z}}{\theta_{1}(z;\tau)}
\sum_{n\in \ZZ}(-1)^n\, 
\frac{q^{\frac{1}{2}n(n+1)}e^{2\pi i n z}}{1-q^ne^{2\pi i z}},\\
\theta_{1}(z;\tau)&\ =\ \sum_{n=-\infty}^{\infty}q^{\frac{1}{2}(n+\frac{1}{2})^2} 
e^{2\pi i (n+\frac{1}{2})(z+\frac{1}{2})}.
\end{align*}
The function $\Sigma(\tau)$ is a Mock modular form and counts the decomposition
of the elliptic genus of a K3 surface into representations of the 
${\cal N}\!\!=\!\!4$ Virasoro algebra.

The Mathieu moonshine phenomenon is that the first five
coefficients appearing in the Fourier expansion of $\Sigma(\tau)$ divided by $2$,
\[
\{45, 231, 770, 2277, 5796\},
\]
are equal to dimensions of irreducible representations of $M_{24}$ and further
coefficients can be written as simple sums of
dimensions of the irreducible representations of $M_{24}$,
for example $13915=3520+10395$.
The reason for this mysterious phenomenon is still unknown.

\smallskip

This observation suggested the existence of a virtual graded
$M_{24}$-module $K=\bigoplus_{n=-1}^{\infty} K_n\, q^{n/8}$ such that for
$n\geq 0$ the $K_n$ are honest $M_{24}$-represen\-tations.
In analogy to the monstrous moonshine case~\cite{CN}, one can consider 
for an element $g$ in the conjugacy class ${\ell X}$ of $M_{24}$ the
so-called McKay-Thompson series 
$$\Sigma_{\ell X}(\tau)=\sum_{n=-1}^{\infty} \tr(g|K_n)\, q^{n/8}.$$ 
In~\cite{{GHV2},{EH}} (cf.\ also \cite{Cheng,GHV,CD}), candidates for the $26$
McKay-Thompson series for the Mathieu moonshine have been proposed.
We note that the McKay-Thompson series for the $M_{24}$ conjugacy classes in the pairs
$(7A, 7B)$, $(14A, 14B)$, $(21A, 21B)$,  $(15A, 15B)$ and $(23A, 23B)$
are equal to each other and we denote these cases shortly by  ${\ell AB}$.
 We list in the appendix the 
$21$ different McKay-Thompson series. Using explicit formulas, it can be shown that
all the McKay-Thompson series have integer coefficients.


Elementary character theory~\cite{serre} implies that the $\Sigma_{\ell X}(\tau)$ together 
uniquely determine the $M_{24}$-module $K$ if it exists.
Terry Gannon~\cite{G} showed that this is indeed the case.
\begin{thm}[Mathieu moonshine module]\label{strongmoonshine}
The McKay-Thompson series as in~\cite{{GHV2},{EH}} determine a virtual graded
$M_{24}$-module $K=\bigoplus_{n=-1}^{\infty} K_n\, q^{n/8}$. For 
$n\geq 0$, the $K_n$ are honest (and not only virtual) $M_{24}$-representations. 
Furthermore, $K_n$ can be decomposed as a direct sum of $M_{24}$-representations of the
form $\lambda\oplus \bar \lambda$ where $\lambda$ is irreducible.
\end{thm}
This implies that if an irreducible representation $\lambda$ is real, i.e.\ $\lambda\cong\bar \lambda$,
then the multiplicity of an irreducible $\lambda$ in $K_n$ is even.

Unlike for the monstrous moonshine case where a vertex operator algebra 
structure was constructed in \cite{FLM}, 
there are yet no known underlying algebraic structures on the Mathieu
moonshine module $K$.

\medskip

To illustrate our main result, we consider the McKay-Thompson series for $7AB$: 
\begin{align*}
\Sigma_{7AB}(\tau)=&\frac{1}{8}
\slfrac{\Bigl(\Sigma(\tau)\eta(\tau)^3-14\,\phi_2^{(7)}(\tau)\Bigr)}{\eta(\tau)^3}\\ 
=&-2{q^{-1/8}}-q^{7/8}+4 q^{31/8}-2 q^{47/8}+2 q^{55/8}-3 q^{63/8}+6 q^{87/8}-6 q^{103/8}\nonumber\\
& -4 q^{119/8}+8 q^{143/8}-6 q^{159/8}+4 q^{167/8}-7 q^{175/8}+12 q^{199/8}
+\cdots.\nonumber
\end{align*}
One observes that the coefficient 
of $q^{n/8}$ in $\Sigma_{7AB}(\tau)$ is odd if 
$n=7 m^2$, where $m$ is odd. In general we show:
\begin{thm}\label{thm:main}
For a conjugacy class $\ell X$ of $M_{24}$, 
the coefficient
of $q^{n/8}$ in $\Sigma_{\ell X}(\tau)$ is odd if and only if 
$\ell X \in \{7AB,\, 14AB,\, 15AB,\, 23AB\}$ and 
$n=\ell m^2$, where $m$ is odd, or 
$\ell X =21AB$ and 
$n=\ell m^2$, where $m$ is odd and 
not divisible by $3$. 
\end{thm}


For congruences of the Fourier coefficients of $\Sigma_{\ell X}(\tau)$ 
for other primes, we refer to the references \cite{{M},{MW},{M2}}. 
One reason for considering the parity of the Fourier coefficients is that
it explains the appearance of certain irreducible representations of $M_{24}$.
The following conjecture was made in~\cite{CDH}, which 
we state for the case of the Mathieu moonshine only.

\begin{conj}[\cite{CDH}, Conj.\ 5.11]\label{conj:CDH}
Let $n=\ell m^2\equiv 7 \pmod{8}$.
Then the $M_{24}$-representation~$K_n$ determined by the coefficients
of $ q^{n/8}$ of the McKay-Thompson series contains the following conjugate pairs of 
irreducible representations:
\vspace{-2mm}
\begin{itemize}
\setlength{\itemsep}{-3pt}
\item For $\ell=7$, one of the pairs $(\chi_3,\chi_4)$, $(\chi_{12},\chi_{13})$ or $(\chi_{15},\chi_{16})$;
\item  for $\ell=15$, the pair  $(\chi_5,\chi_6)$;
\item  for $\ell=23$, the pair  $(\chi_{10},\chi_{11})$.
\end{itemize}
\end{conj}
Here, $\chi_i$ where $1\leq i \leq 26$ denotes the $i$-th irreducible representation as listed
in the ATLAS~\cite{atlas}.

\medskip

The paper is organized as follows.
In Section \ref{sec:odd}, we study the cases 
$7AB$, $14AB$, $15AB$, $21AB$, $23AB$. 
In Section \ref{sec:others}, we study the remaining cases. 
In the final section, we use Theorem~\ref{thm:main} to prove 
Conjecture~\ref{conj:CDH}.

\bigskip

\noindent{\it Acknowledgments.} \ 
The first two authors like to thank the Hausdorff Research Institute of Mathematics for the 
hospitality received during their stay for the Mathematical Physics trimester program.
We also like to thank M.~Cheng, T.~Gannon and S.~Hohenegger for discussions.
Finally, we are grateful to the referees for the suggested improvements of the exposition.

\section{The non-even cases}\label{sec:odd}

In this section, we prove Theorem~\ref{thm:main} for the cases 
$7AB$, $14AB$, $15AB$, $21AB$ and $23AB$. 

\smallskip

The strategy is to obtain the
parity properties of $\Sigma_{\ell X}(\tau)$ from the parity properties of a modular
form for a group $\Gamma_0(N)$ which in turn can be proven by an application of 
Sturm's theorem~\cite{Sturm}. The theorem allows to obtain a divisibility property of
the Fourier coefficients of a modular form of weight $k$ for $\Gamma_0(N)$ by veryfing 
this property only for the first~$n$ coefficients where $n$ is an explicit given bound 
depending only on $N$ and $k$,
\[
n\geq \frac{k}{12}[{\rm SL}_2(\ZZ):\Gamma_0(N)]. 
\] 


We let
\[\eta(\tau)=q^{1/24}\prod_{m=1}^{\infty}(1-q^m)\]
be the Dedekind $\eta$-function and consider
for $N\geq 2$ the Eisenstein series
\begin{eqnarray}\label{eq:phiN}
\phi_2^{(N)}(\tau)&=&\frac{24}{N-1}\,q\,\frac{d}{dq}
\log\left(\frac{\eta(N\tau)}{\eta(\tau)}\right)\\ \nonumber
&=&1+\frac{24}{N-1}\sum_{k=1}^{\infty}\sigma_1(k)(q^k-Nq^{Nk})
\end{eqnarray} of weight $2$ for $\Gamma_0(N)$.


\begin{proof}[\it Proof of Theorem~\ref{thm:main} for the case $7AB$]

Define the function $f_{m}(\tau)$ as follows: 
\begin{align*}
f_m(\tau)&=\frac{1}{4}\Bigl(\vartheta_3(\frac{m\tau}{8})-\vartheta_4(\frac{m\tau}{8})\Bigr) \end{align*}
where $\vartheta_3(\tau)=1+\sum_{m=1}^{\infty}2q^{m^2}$ 
and $\vartheta_4(\tau)=1+\sum_{m=1}^{\infty}2(-q)^{m^2}$.
Then 
\begin{align*}
f_7(\tau)&=\frac{1}{4}\Bigl(\vartheta_3(\frac{7\tau}{8})-\vartheta_4(\frac{7\tau}{8})\Bigr)=
q^{7/8}+q^{63/8}+q^{175/8}+q^{343/8}+\cdots 
\end{align*}
We call $f_7(\tau)$ the ``parity function.'' 

We have to show that all coefficients of 
$\Sigma_{7AB}(\tau)+f_7(\tau)$ are even since 
the coefficient of $q^{n/8}$ in $f_7(\tau)$ is odd 
if and only if $n=7 m^2$, where $m$ is odd.
This follows if all coefficients of 
$(\Sigma_{7AB}(\tau)+f_7(\tau))\eta(\tau)^3$ are even since 
$\eta(\tau)^{-3}$ has integral coefficients.

\medskip

Let 
\[
E_2(\tau)=1-24\sum_{m=1}^{\infty}\sigma_1(m)q^m. 
\]
By {\cite[p.~43, Example 4]{DMZ}}, we have 
\begin{align}\label{eqn:DMZ}
\Sigma(\tau)=
-2\slfrac{\Bigl(E_2(\tau)+24\displaystyle \sum_{n=1}^{\infty}\frac{(-1)^n n q^{\frac{1}{2}n(n+1)}}{1-q^n}\Bigr)}{\eta(\tau)^3}.
\end{align}
Thus 
\begin{equation*}\label{eqn:7AB}
(\Sigma_{7AB}(\tau)+f_7(\tau))\eta(\tau)^3 =\frac{1}{8}
\left(\Sigma(\tau)\eta(\tau)^3-14\,\phi_2^{(7)}(\tau)
\right)+ \eta(\tau)^3f_7(\tau) 
\end{equation*}
\begin{equation*}
\qquad\qquad=-\frac{1}{4}\Bigl(
E_2(\tau)+24\displaystyle \sum_{n=1}^{\infty}\frac{(-1)^n n q^{\frac{1}{2}n(n+1)}}{1-q^n}+
7\phi_2^{(7)}(\tau)\Bigr)+ \eta(\tau)^3f_7(\tau) . 
\end{equation*}
First, we observe that the constant term is even. 
Since all coefficients, except for the constant term, of the function  
\[
-\frac{1}{4}\Bigl(
E_2(\tau)+24\displaystyle \sum_{n=1}^{\infty}\frac{(-1)^n n q^{\frac{1}{2}n(n+1)}}{1-q^n}\Bigr)
\]
are even,  it is enough to show that the coefficients of 
the following function, again except for the constant term, 
\[
-\frac{7}{4}\phi_2^{(7)}(\tau)+ \eta(\tau)^3f_7(\tau)=
-\frac{7}{4}-6 q-24 q^2-28 q^3-44 q^4-42 q^5+\cdots.
\]
are even.

Define the ``correction function'' 
\begin{align*}
\frac{7}{4}\vartheta_3(\tau)^4=\frac{7}{4}+14 q+42 q^2+56 q^3+42 q^4+84 q^5
+\cdots .
\end{align*}
The coefficients of 
$\frac{7}{4}\vartheta_3(\tau)^4$ without the constant term are even, 
hence it is enough to show that 
the coefficients of the function 
\begin{align*}
-\frac{7}{4}&\phi_2^{(7)}(\tau)+ \eta(\tau)^3f_7(\tau)+\frac{7}{4}\vartheta_3(\tau)^4=8 q+18 q^2+28 q^3-2 q^4+42q^5+\cdots
\end{align*}
are even.
We are going to prove this using Sturm's theorem. The theta functions 
$\vartheta_3(\tau)$ and $\vartheta_4(\tau)$ can be expressed as a quotient of $\eta$-functions, namely 
\begin{align*}\label{eqn:theta}
\vartheta_3(\tau)=\frac{\eta(2\tau)^5}{\eta(\tau)^2\eta(4\tau)^2} \qquad\text{and}\qquad
\vartheta_4(\tau)=\frac{\eta(\tau)^2}{\eta(2\tau)} .
\end{align*}
It follows using \cite[Theorem 1.64]{Ono} that $\eta(8\tau)^3f_7(8\tau)$ is a 
modular form of weight $2$ for $\Gamma_0(448)$ and hence
\begin{align*}
\sum_{m=1}^{\infty}a_{7AB}(m)q^m & := -\frac{7}{4}\phi_2^{(7)}(8\tau)+ \eta(8\tau)^3f_7(8\tau)
+\frac{7}{4}\vartheta_3(8\tau)^4\\
&=8 q^8+18 q^{16}+28 q^{24}+\cdots 
\end{align*}
is also a 
modular form of weight $2$ for $\Gamma_0(448)$. 
Using Sturm's theorem \cite{Sturm} (see also \cite[Theorem 2.58]{Ono}) and 
the fact that $[{\rm SL}_2(\ZZ):\Gamma_0(448)]=768$, 
the computer verification that 
$a_{7AB}(m) \equiv 0 \pmod{2}$ for $m \leq 129$ shows 
$a_{7AB}(m) \equiv 0 \pmod{2}$ for all $m$. 
 
This completes the proof of the case $7AB$. 
\end{proof}

\medskip

Since the other cases can be handled in complete analogy, we 
collect the relevant information in Table \ref{Tab:NN1}.
For the definition of $\Sigma_{23AB}(\tau)$ in Appendix~A, we use functions 
$f_{23,1}(\tau)$ and $f_{23,2}(\tau)$ which are modular forms for $\Gamma_0(23)$,
explicitly given in \cite[Appendix A.1]{EH2}. 

\begin{table}[thb]
\renewcommand{\arraystretch}{1.2}
\caption{Data for the proofs in Section \ref{sec:odd}}
\label{Tab:NN1}
\begin{center}
{\footnotesize
\begin{tabular}{c|ccccc} 
$\ell X$ & parity function & 
correction function& $\Gamma$&$[{\rm SL}_2(\ZZ):\Gamma]$& Sturm bound \\
\hline
$7AB$ & $f_7(\tau)$ & $\frac{7}{4}\vartheta_3(\tau)^4$ & $\Gamma_0(448)$ & $768$ & $129$ \\
$14AB$ & $f_7(\tau)$&  $\frac{23}{12}\phi_2^{(2)}(\tau)$ & $\Gamma_0(448)$ & $768$ & $129$ \\
$15AB$ & $f_{15}(\tau)$ &$\frac{23}{12}\phi_2^{(2)}(\tau)$ & $\Gamma_0(960)$ & $2304$ & $385$ \\
$21AB$ & $f_7(\tau)$-$f_{63}(\tau)$ &$2\vartheta_3(\tau)^4$ & $\Gamma_0(4032)$ & $9216$ & $1537$ \\
$23AB$ & $f_{23}(\tau)$& $\frac{23}{12}\phi_2^{(2)}(\tau)$ & $\Gamma_0(1472)$ & $2304$ & $385$ \\
\end{tabular}
}
\end{center}
\renewcommand{\arraystretch}{1.0}
\end{table}

\section{The even cases}\label{sec:others}

In this section, we prove that all coefficients of the McKay-Thompson series 
for the remaining conjugacy classes of $M_{24}$ are divisible by two.

\smallskip

We remark that the case $1A$ is clear since $\Sigma_{1A}(\tau)=\Sigma(\tau)$ and equation~(\ref{eqn:DMZ}).
It is trivial to see that for the cases 
$2B$,  $3B$, $4A$, $4C$,  $6B$,  $10A$, $12A$, $12B$,
the coefficients of $\Sigma_{\ell X}(\tau)$ are even because 
$\Sigma_{\ell X}(\tau)$ is $-2$ times an $\eta\mbox{-product}$ (see the appendix). 


\subsection{The cases $2A$, $3A$, $4B$, $5A$, $6A$, $8A$}\label{sec:others1}

We give the detailed proof for the case of $2A$.
\begin{proof}
The McKay-Thompson series for $2A$ is 
\begin{align*}
\displaystyle
\Sigma_{2A}(\tau)&=
\frac{1}{3}\slfrac{\Bigl(
\Sigma(\tau)\eta(\tau)^3 - 4 \phi_2^{(2)}(\tau)\Bigr)}{\eta(\tau)^3}\\ 
&=-2q^{-1/8}-6 q^{7/8}+14 q^{15/8}-28 q^{23/8}+42 q^{31/8}+\cdots.
\end{align*}
The coefficients of $\Sigma(\tau)\eta(\tau)^3$, 
except for the constant term, are divisible by $24$ 
(cf.~(\ref{eqn:DMZ})). 
Moreover, also the coefficients of
$\phi_2^{(2)}(\tau)$, again except for the constant term, are divisible by $24$. 
The constant term of the function 
$\Sigma(\tau)\eta(\tau)^3 - 4 \phi_2^{(2)}(\tau)$ is $6$. 
Therefore, the coefficients of $\Sigma_{2A}(\tau)$ are 
divisible by two. 
\begin{table}[thb]
\renewcommand{\arraystretch}{1.2}
\caption{Data for the proofs in Section \ref{sec:others1}}
\label{Tab:NN2}
\begin{center}
{\footnotesize
\begin{tabular}{c|cc} 
$\ell X$ & $N$ & 
divisor of $\phi_2^{(N)}(\tau)-1$ \\
\hline
$2A$ & $2$ & $24$ \\
$3A$ & $3$ & $12$ \\
$5A$ & $5$ &$6$ \\
$4B$ & $2$, $4$ &$24$, $8$ \\
$8A$ & $4$, $8$ & $8$, $24/7$  \\
$6A$ & $2$, $3$, $6$ & $24$, $12$, $24/5$ \\
\end{tabular}
}
\end{center}
\renewcommand{\arraystretch}{1.0}
\end{table}
\end{proof}

The proof for the other cases is analogous.
The relevant information can be read off from Table~\ref{Tab:NN2} together with the appendix.


\subsection{The case $11A$}

Finally, it remains to show that 
the Fourier coefficients of the McKay-Thompson series for the case $11A$ are 
divisible by two. Because of the extra term 
$\frac{264}{5}(\eta(\tau)\eta(11\tau))^2$
in $\Sigma_{11A}(\tau)$, it is not obvious that the coefficients except for the 
constant term in the numerator for $\Sigma_{11A}(\tau)$ as given in the appendix
are divisible by $24$. 
Therefore, we proceed similar to the cases in Section \ref{sec:others1}.
\begin{proof}
The McKay-Thompson series for $11A$ is 
\begin{align*}
\displaystyle
\Sigma_{11A}(\tau)=\frac{1}{12} 
\slfrac{\Bigl(
\Sigma(\tau) \eta(\tau)^3- 22\,\phi_2^{(11)}(\tau)
+\frac{264}{5}(\eta(\tau)\eta(11\tau))^2\Bigr)}{\eta(\tau)^3}.
 \end{align*}
The coefficients of $\Sigma(\tau)\eta(\tau)^3$, except for the constant term 
are divisible by $24$. We need to prove that the coefficients of the function:
\[
22\,\phi_2^{(11)}(\tau)-\frac{264}{5}(\eta(\tau)\eta(11\tau))^2\\
=22+264 q^2+264 q^3+264 q^4
%
+\cdots 
\]
are also (except for the constant term) divisible by $24$.
The coefficients of $22\phi_2^{(2)}(\tau)$ 
(except for the constant term) are divisible by $24$, 
hence it is enough to show that the coefficients of the function 
\begin{align*}
\sum_{m=1}^{\infty}&a_{11A}(m)q^m := 22\phi_2^{(11)}(\tau)
-\frac{264}{5}(\eta(\tau)\eta(11\tau))^2
-22\,\phi_2^{(2)}(\tau)\\
&=-528 q-264 q^2-1848 q^3-264 q^4-2904 q^5-1584 q^6-3696 q^7
+\cdots
\end{align*}
are divisible by $24$.
Note that this function is a modular form for 
$\Gamma_0(22)$ \cite[p.~130, Proposition 19]{K}. 
Using Sturm's theorem \cite[Theorem 2.58]{Ono} again and 
the fact that $[{\rm SL}_2(\ZZ):\Gamma_0(22)]=36$, 
the verification that 
$a_{11A}(m) \equiv 0 \pmod{24}$ for $m \leq 7$ shows 
$a_{11A}(m) \equiv 0 \pmod{24}$ for all $m$. 
Therefore, the coefficients of $\Sigma_{11A}(\tau)$ are 
divisible by two. 
\end{proof}

\medskip

Sections \ref{sec:odd} and \ref{sec:others} together 
prove Theorem~\ref{thm:main}. 


\section{Multiplicities of Irreducibles}

Recall that $\chi_i$ denotes the $i$-th irreducible $M_{24}$-representation as in~\cite{atlas}.
We will use Theorem~\ref{thm:main} to prove the following result about
the multiplicities of the $M_{24}$-representation $\lambda \oplus\bar\lambda$ with irreducible $\lambda$
inside the Mathieu moonshine module constituents $K_n$:
\begin{thm}\label{application}
Let $n=\ell m^2\equiv 7 \pmod{8}$ and let $K_n$ be the degree $n$-part of the Mathieu moonshine 
module $K$.
Then the following numbers are odd (and therefore positive):
\begin{itemize}
\setlength{\itemsep}{-3pt}
\item For $\ell=7$, the multiplicity of  $\chi_3\oplus \chi_4$ in $K_n$
plus the multiplicity of $\chi_{12}\oplus \chi_{13}$ in $K_n$;
\item for $\ell=7$ and $m$ divisible by $3$, the multiplicity of $\chi_{15} \oplus \chi_{16}$ in $K_n$;
\item  for $\ell=15$, the multiplicity of $\chi_5\oplus \chi_6$ in  $K_n$;
\item  for $\ell=23$, the multiplicity of  $\chi_{10}\oplus \chi_{11}$ in $K_n$.
\end{itemize}
\end{thm}

\begin{proof}
Let 
\begin{equation}\label{decomp}
K_n=\bigoplus_{i=1}^{26} m_{\chi_i}\,  \chi_i
\end{equation}
be the decomposition of the $M_{24}$-representation
$K_n$ into its irreducible constituents, i.e.\ $m_{\chi_i}$ is the multiplicity in which
$\chi_i$ occurs in $K_n$. 
Taking on both sides of~(\ref{decomp}) the trace for an element $g\in M_{24}$ we have
\begin{equation}\label{trace}
\tr(g|K_n)=\sum_{i=1}^{26} m_{\chi_i}\,   \tr(g|\chi_i)
\end{equation}
and $\tr(g|K_n)$ is the coefficient of $q^{n/8}$ in $\Sigma_{\ell X}(\tau)$ with
${\ell X}$ the conjugacy class of $g$.
We note that if $K_n$ is non-zero then $n$ is odd, i.e.\ we only need to consider
the cases $n=\ell m^2$ with $m$ odd.

\smallskip

First, we consider the cases $\ell=15$ and  $\ell=23$ and let $g$ be an element of type $15A$ or $23A$, 
respectively.  If $n$ is of the form $\ell m^2$ with $m$ odd, the left-hand side of~(\ref{trace}) 
is odd by Theorem~\ref{thm:main}. 
For the right-hand side
an inspection of the character table of $M_{24}$~\cite{atlas} shows that if $\lambda\not=\bar\lambda$ one has that
$\tr (g|\lambda)=\tr (g| \bar \lambda)$ is integral unless $(\lambda,\bar\lambda)=(\chi_5,\chi_6)$ or 
$(\chi_{10},\chi_{11})$,
respectively, in which case $\tr (g|\lambda+\bar\lambda)=-1$. Using Theorem~\ref{strongmoonshine}, 
it follows that $m_\lambda=m_{\bar\lambda}$ has then to be odd.

For the cases  $\ell=7$, we let $g$ be an element of type $7A$. Here, only
the characters $(\chi_3,\chi_4)$ or $(\chi_{12},\chi_{13})$ can provide an odd contribution 
to the right-hand side of (\ref{trace}) and so in total an odd number of those pairs has to appear.
For $m$ divisible by $3$, we take in addition an element $g$ of type $21A$. 
Here $(\chi_3,\chi_4)$, $(\chi_{12},\chi_{13})$ and $(\chi_{15},\chi_{16})$ provide an odd contribution. 
Since the total number of pairs of type $(\chi_3,\chi_4)$ and $(\chi_{12},\chi_{13})$ is odd and 
by Theorem \ref{thm:main} the  left-hand side of (\ref{trace}) is even, it follows that 
the number of pairs $(\chi_{15},\chi_{16})$ is odd, too.
\end{proof}

\smallskip

It is clear that Theorem~\ref{application} implies Conjecture~\ref{conj:CDH}.

\medskip

The argument in the proof can also be applied to the elements $g$ of type ${\ell X}$ as
studied in Section~\ref{sec:others}. One directly recovers Theorem~\ref{thm:main} for those
cases. However, Theorem~\ref{strongmoonshine} as stated is a refinement which uses parts of
our original arguments in its proof.

\smallskip

We remark that it is likely that the methods in proving Theorem~\ref{thm:main} and Theorem~\ref{application} 
could also be applied to the other cases of umbral moonshine.


\appendix

\section{McKay-Thompson series}

We list all McKay-Thompson series using the ATLAS~\cite{atlas} names for the 
conjugacy classes of $M_{24}$. We refer to Section~\ref{sec:odd} for 
the functions used.
{\small
\renewcommand{\arraystretch}{1.65}
$
\begin{array}{rl}
{\bf 1A:}&\
\Sigma_{1A}(\tau)=
\Sigma(\tau)
\\
{\bf 2A:}&\ 
\Sigma_{2A}(\tau)=
\frac{1}{3} 
\slfrac{\left(
\Sigma(\tau)\eta(\tau)^3 - 4\,\phi_2^{(2)}(\tau)
\right)}{\eta(\tau)^3} \\
{\bf 2B:}& \
\Sigma_{2B}(\tau)=
-2\slfrac{\eta(\tau)^5}
{\eta(2\tau)^4}\\ 
{\bf 3A:}& \
\Sigma_{3A}(\tau)=
\frac{1}{4} 
\slfrac{\left(
\Sigma(\tau)\eta(\tau)^3 - 6\,\phi_2^{(3)}(\tau)
\right)}{\eta(\tau)^3} \\
{\bf 3B:}& \
\Sigma_{3B}(\tau)=
-2\slfrac{\eta(\tau)^3}
{\eta(3\tau)^2}\\ 
{\bf 4A:}& \
\Sigma_{4A}(\tau)=
-2\slfrac{\eta(2\tau)^8}
{\bigl(\eta(\tau)^3\eta(4\tau)^4\bigr)} \\ 
{\bf 4B:}& \
\Sigma_{4B}(\tau)=
\frac{1}{6} 
\slfrac{\left(
\Sigma(\tau)\eta(\tau)^3 +2\phi_2^{(2)}(\tau)-12\,\phi_2^{(4)}(\tau)
\right)}{\eta(\tau)^3}\\ 
{\bf 4C:}& \
\Sigma_{4C}(\tau)=
-2\slfrac{\eta(\tau)\eta(2\tau)^2}
{\eta(4\tau)^2}\\ 
{\bf 5A:}&\ 
\Sigma_{5A}(\tau)=
\frac{1}{6} 
\slfrac{\left(
\Sigma(\tau)\eta(\tau)^3 - 10\, \phi_2^{(5)}(\tau)
\right)}{\eta(\tau)^3}\\
{\bf 6A:}& \
\Sigma_{6A}(\tau)=
\frac{1}{12} 
\slfrac{\left(
\Sigma(\tau)\eta(\tau)^3 +2\,\phi_2^{(2)}(\tau)+6\,\phi_2^{(3)}(\tau)
-30\,\phi_2^{(6)}(\tau)
\right)}{\eta(\tau)^3} \\
{\bf 6B:}& \
\Sigma_{6B}(\tau)=
-2\slfrac{\eta(2\tau)^2\eta(3\tau)^2}
{\bigl(\eta(\tau)\eta(6\tau)^2\bigr)}\\
{\bf 7AB:}& \
\Sigma_{7AB}(\tau)=
\frac{1}{8}
\slfrac{\left(\Sigma(\tau)\eta(\tau)^3-14\,\phi_2^{(7)}(\tau)
\right)}{\eta(\tau)^3} \\ 
{\bf 8A:}& \
\Sigma_{8A}(\tau)=
\frac{1}{12} 
\slfrac{\left(
\Sigma(\tau)\eta(\tau)^3 +6\,\phi_2^{(4)}(\tau)-28\,\phi_2^{(8)}(\tau)
\right)}{\eta(\tau)^3} \\ 
{\bf 10A:}& \
\Sigma_{10A}(\tau)=
-2\slfrac{\eta(2\tau)\eta(5\tau)}
{\eta(10\tau)}\\ 
{\bf 11A:}& \
\Sigma_{11A}(\tau)=
\frac{1}{12} 
\slfrac{\left(
\Sigma(\tau)\eta(\tau)^3 - 22\,\phi_2^{(11)}(\tau)
+\frac{264}{5}\,(\eta(\tau)\eta(11\tau))^2
\right)}{\eta(\tau)^3} \\
{\bf 12A:}& \
\Sigma_{12A}(\tau)=
-2\slfrac{\eta(4\tau)^2\eta(6\tau)^3}
{\bigl(\eta(2\tau)\eta(3\tau)\eta(12\tau)^2\bigr)} 
\\
{\bf 12B:}& \
\Sigma_{12B}(\tau)=
-2\slfrac{\eta(\tau)\eta(4\tau)\eta(6\tau)}
{\bigl(\eta(2\tau)\eta(12\tau)\bigr)}
\qquad\qquad\qquad\qquad\qquad\qquad\qquad\qquad\\
{\bf 14AB:}& \
\Sigma_{14AB}(\tau)=
\frac{1}{24}\,
\Bigl(\Sigma(\tau)\eta(\tau)^3
+\frac{2}{3}\,\phi_2^{(2)}(\tau)
+14\phi_2^{(7)}(\tau)
-\frac{182}{3}\,\phi_2^{(14)}(\tau)+\Bigr.
\\
&\qquad\qquad\qquad\quad
\slfrac{\Bigl.+112\,\eta(\tau)\eta(2\tau)\eta(7\tau)\eta(14\tau)\Bigr)
}{\eta(\tau)^3}   \\
{\bf 15AB:}& \
\Sigma_{15AB}(\tau)=
\frac{1}{24}\,
\Bigl(
\Sigma(\tau)\eta(\tau)^3+\frac{3}{2}\,\phi_2^{(3)}(\tau)+5\, \phi_2^{(5)}(\tau) 
- \frac{105}{2}\, \phi_2^{(15)}(\tau)+\Bigr.
\\
&\qquad\qquad\qquad\quad
\slfrac{\Bigl. +90\, \eta(\tau) \eta(3\tau) \eta(5\tau) \eta(15\tau)
\Bigr)}{\eta(\tau)^3}  \\  
{\bf 21AB:}& \
\Sigma_{21AB}(\tau)=
-\frac{1}{3}
\left(
\slfrac{7\,\eta(7\tau)^3}{\bigl(\eta(3\tau)\eta(21\tau)\bigr)}
-\slfrac{\eta(\tau)^3}{\eta(3\tau)^2}
\right) \\ 
{\bf 23AB:}& \
\Sigma_{23AB}(\tau)=
\frac{1}{24}\,
\slfrac{\Bigl(
\Sigma(\tau) \eta(\tau)^3-  46\, \phi_2^{(23)}(\tau) +
 \frac{276}{11} \, f_{23,1}(\tau) +\frac{1932}{11} f_{23,2}(\tau)
\Bigr)}{\eta(\tau)^3}  \\ 
\end{array}
$
}


\end{document}